\documentclass{article}[12pt]
\usepackage{authblk}
\usepackage{amsfonts,amsmath,amsthm,amssymb,mathtools}
\usepackage[pdftex]{graphicx}
\usepackage{textcomp}
\usepackage{pgf,tikz}
\usetikzlibrary{arrows}
\usepackage{dsfont}
\usepackage{enumerate}
\usepackage{bbm}
\usepackage[ruled,vlined]{algorithm2e}
\usepackage[utf8x]{inputenc} 
\usepackage{bbold}
\usepackage{breqn}
\usepackage{tcolorbox}
\usepackage{bm}
\usepackage{color}
\usepackage{listings}
\usepackage{graphicx}
\usepackage[titletoc,title]{appendix}
\usepackage{float}

\usepackage{fancybox}
\usepackage{shadow}
\usepackage[all]{xy}
\usepackage{url}
\usepackage{rotating}

\usepackage{amsthm}
\usepackage{amscd}

\usepackage{graphicx}%
\usepackage{fancyhdr}
\usepackage{geometry}
\usepackage{mathrsfs}

\usepackage{hyperref}

\def\ignore#1{{}}

\newtheorem{theorem}{Theorem}

\newtheorem{remark}{Remark}
\newcommand{\p}{\mathbbm{P}}
\newcommand{\e}{\mathbbm{E}}
\newcommand{\N}{\mathbbm{N}}

\newcommand{\R}{\mathbbm{R}}

\usepackage{authblk}
\title{Escape from parsimony of a double-cut-and-join \\genome evolution process}
\author[1]{Mona Meghdari Miardan}
\author[2]{Arash Jamshidpey}
\author[1]{David Sankoff}
\affil[1]{Department of Mathematics and Statistics, University of Ottawa, Ottawa, ON K1N 6N5, Canada}
\affil[2]{Department of Mathematics, Columbia University, 2990 Broadway, New York, NY 10027, USA}
\date{}                     
\begin{document}
\date{September 27, 2021}
\maketitle
\begin{abstract}
We analyze models of genome evolution based on both restricted and unrestricted double-cut-and-join (DCJ) operations.  We compare  the number of operations along the evolutionary trajectory to the DCJ distance of the genome from its ancestor at each step, and determine at what point they diverge: the process escapes from parsimony. Adapting the method developed by Berestycki and Durret \cite{bd-2006}, we estimate the number of cycles in the breakpoint graph of a random genome at time $t$ and its ancestral genome by the number of tree components of an Erd\"os-R\'enyi random graph constructed from the model of evolution. In both models, the process on a genome of size $n$ is bound to its parsimonious estimate up to $t\approx n/2$ steps.
\end{abstract}
\section{Introduction} The introduction of the ``double-cut-and-join" (DCJ) methods, by \cite{yaf-2005}, and in a slightly different formulation  by \cite{bms-2006}, greatly invigorated the field of genome distance algorithmics.  DCJ incorporated the operation of  interchanging blocks of genes between arbitrary regions of a chromosome (including ``block transposition" of neighbouring regions) to the two previously studied operations of reversing of a block of genes in a chromosome (``inversion") and interchanging the prefixes or suffixes of  two chromosomes (``reciprocal translocation", which formally includes chromosome fusion and fission).   As a result, the exact DCJ  distance is easily calculated in linear time, avoiding the cumbersome work necessitated by the ingenious solution (called ``HP") of \cite{HP-95a} and \cite{HP-95b} a decade earlier, where only reversals and reciprocal translocations were considered.

Edit distances for comparing genomes, like DCJ and HP, minimize the number of steps, chosen from a small repertoire of operations, to transform one genome into another.  A mistaken comment sometimes heard, whether approving or critical, is that these methods assume that ``evolution takes the shortest path".  The methods, however, make no assumption at all about the preferences of the evolutionary process;  they are simply tools for inferring evolutionary history.

The minimization criterion basic to edit distances makes it a maximum parsimony method.  As such it is model-free, unlike probability-based models.  Nevertheless, under any reasonable model of evolution, calculations or simulations show that for a genome that has only moderately evolved, DCJ or similar methods inevitably recover the true number of steps that actually occurred.   In this sense, evolution has indeed ``followed" a most parsimonious path, but this is only an inferential result, not an evolutionary tendency.  We say that the evolution is``parsimony-bound" during this initial period.

As a genome diverges more substantially from its initial state, however, parsimony methods will eventually find a shorter path than the one actually traversed by the evolving genome.   Parsimony underestimates the length of the evolutionary trajectory.  When this occurs, we can say that evolution ``escapes from parsimony".

The nature of this escape and the point at which it happens depends on the stochastic process modeling evolution and the distance measure comparing points on the sample path to the starting point.  Little mathematical work has been done relevant to this direct question, although there has been much research on probabilistic modeling of evolution using DCJ operations, as reviewed in \cite{bgt-2015}.  In this paper, we model the evolution of genomes as a Markov jump process on the space of genomes, and use DCJ as a tool for tracking it.

A basic result in this field is due to Berestycki and Durrett \cite{bd-2006}, who found that escape from parsimony for a random transposition (pairwise exchange) process does not occur before $n/2$ steps, where $n$ is the number of genes in the genome. 

\section{Genomes and DCJ operations}
In the absence of duplication, genes or markers in a genome can be represented by different positive integers. We denote by $\mathcal{G}_{n,k}$ the space of all genomes with $n$ specific genes (or markers), denoted by $1,\cdots,n$, exactly $k$ linear chromosomes, and a number of circular chromosomes (possibly 0). For the purpose of this paper, it is most convenient to represent a genome $G\in\mathcal{G}_{n,k}$ as an alternating graph $G=(V(G),E(G))$ on $2(n+k)$ vertices, $2n$ of which are of degree $2$ and the rest are of degree $1$,as in Figure \ref{graph}. More precisely, we denote by $+i$ and $-i$ the extremities (head and tail) of gene (or marker) $i\in [n]:=\{1,\cdots,n\}$, and denote by $\tau_j$, $j\in [2k]$, the telomeres of $k$ linear chromosomes, i.e. $$V(G)=\{\pm 1,\cdots,\pm n\}\cup \{\tau_1,\cdots,\tau_{2k}\}.$$ On the other hand, let $E_1(G)$ be  invariant for all genomes $G$, and define it to be the set of edges $\{-i,+i\}$, for $i\in [n]$, and let $E_2(G)$ vary for different genomes $G$ and be defined as a perfect matching on $V(G)$. Then, $E(G)$  consists of the edges in $E_1(G)$ and $E_2(G)$, where multiple edges are allowed. We call $E_1(G)$ and $E_2(G)$ the set of genes and set of adjacencies of genome $G$, respectively. We notice that $|E_2(G)|=n+k$, for any $G\in \mathcal G_{n,k}$. We suppose the edges in $E_1(G)$ and $E_2(G)$ are dotted and solid, respectively. The solid edges which are incident to the telomeric vertices $\{\tau_1,\cdots,\tau_{2k}\}$ are called telomeric edges, and those vertices in $\{\pm 1,\cdots,\pm n\}$ which are adjacent to $\{\tau_1,\cdots,\tau_{2k}\}$ are called telomeric gene extremities.  It is then clear that the graph $G$ is a union of disjoint alternating paths (linear chromosomes) and cycles (circular chromosomes). Indeed, $deg(+i)=deg(-i)=2$ and $deg(\tau_{j})=1$, for any $i\in [n]$ and $j\in [2k]$. Note that, under this definition, a linear chromosome may be null (i.e. with no genes). In this case, the null chromosome only contains two telomeres $\tau_i, \tau_j$ and their connecting edge $\{\tau_i,\tau_j\}$.\\
A natural orientation can be assigned to the connected components (chromosomes)  of each genome in $\mathcal{G}_{n,k}$. To establish this, recall that an orientation of a graph $G$ is a digraph whose underlying graph is $G$. In other words, we can assign a direction to each edge $\{u,v\}\in E(G)$ by assigning exactly one of the two possible arcs $(u,v)$ or $(v,u)$ to it. An orientation of $G$ is then a digraph obtained from $G$ by assigning a direction to each of its edges. Clearly, there are $2^{2(n+k)}$ orientations for $G\in \mathcal{G}_{n,k}$. Let $C$ be a cycle component, and let $P$ be a path component of $G$. A direction of $C$, namely $\overrightarrow{C}$, is an orientation for which $\overrightarrow{C}$ is a directed cycle, that is $deg_+(v)=deg_-(v)=1$, for $v\in V(C)$, where $deg_+(v)$ and $deg_-(v)$ denote the outdegree and indegree of $v$. Similarly, one can give a direction to a linear chromosome $P$, by choosing one of its telomeres as a starting point. To be more precise,  an orientation $\overrightarrow{P}$ is called a direction for $P$, if it is a directed path.  Hence, there are exactly two directions for each component of $G$. Denoting by $\kappa(G)$ the number of connected components of $G$, this implies that the number of directions of $G$ is $2^{\kappa(G)}$. Now, let $x$ be the gene with the smallest label in a circular chromosome $C$ (or a linear chromosome $P$, respectively). The standard direction is the unique direction of $C$ ( $P$, respectively) for which the direction of $\{-x,+x\}$ is given by  arc $(-x,+x)$. Moreover, the standard direction of $G$ is the unique orientation of that for which each component is assigned to its standard direction. We henceforth assume that each genome $G\in \mathcal{G}_{n,k}$ is furnished with its standard direction. \\

\begin{figure}[!h]
            \centering
            \includegraphics[width=.85\textwidth]{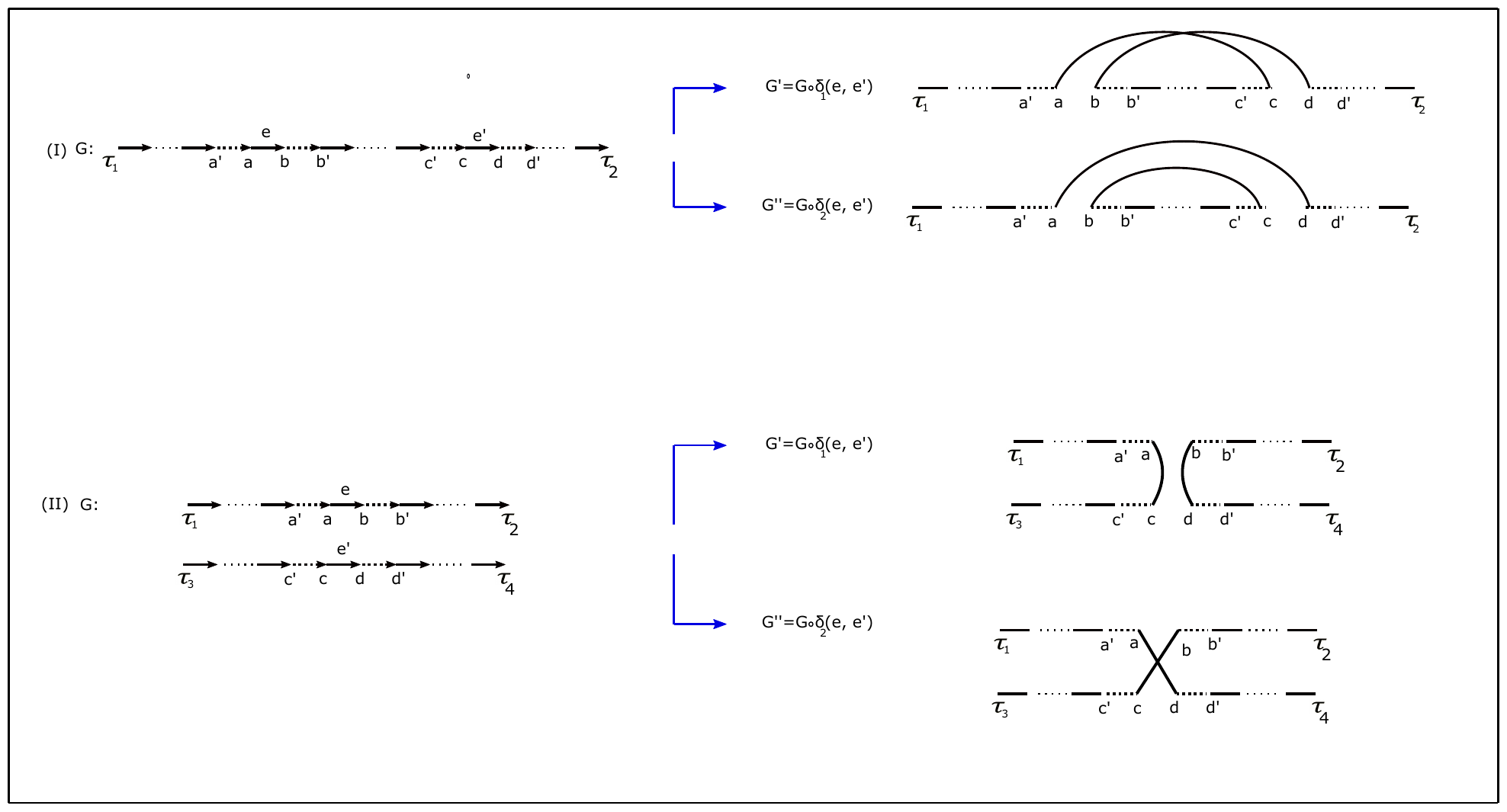}
            \caption{DCJ operations.}
            \label{graph}
        \end{figure}
A double-cut-and-join (DCJ) operation is a genomic operation that cuts two given adjacencies (edges) of a genome and rejoins the free extremities of the involved genes in one of the two possible ways (Figure \ref{graph}). Two types of the DCJ operations are defined as follows. For $G\in \mathcal{G}_{n,k}$, and $e=(a,b),e'=(c,d)\in E_2(G)$, a first-type DCJ (DCJ-1 or $\delta_1$) operation, denoted by $\delta_1(e,e')$ is defined by $G'=G \circ \delta_1(e,e')$, where $V(G')=V(G)$, $E_1(G')=E_1(G)$, and $$E_2(G')=E_2(G)\cup\{(a,c),(b,d)\}\setminus\{e,e'\}.$$ Similarly, a second-type DCJ (DCJ-2 or $\delta_2$) operation on $G$ is defined by $G''=G\circ \delta_2(e,e')$, where $V(G'')=V(G)$, $E_1(G'')=E_1(G)$, and $$E_2(G'')=E_2(G)\cup\{(a,d),(b,c)\}\setminus\{e,e'\}.$$ Note that for $e,e'$ taken from the same chromosome, $\delta_1(e,e')$ is an inversion (reversal) operation, as shown in Figure \ref{graph}-$I$. From the definition, $\mathcal{G}_{n,k}$ is closed under the DCJ operations, and the DCJ distance between two genomes $G_1,G_2\in \mathcal{G}_{n,k}$, $d(G_1,G_2)$, is the minimum number of DCJ operations required to transform $G_1$ into $G_2$ or vice-versa. Equivalently, this distance can be interpreted as the length of the shortest path on the graph $\mathscr{G}_{n,k}$ with the set of vertices $\mathcal{G}_{n,k}$ where each pair of vertices are connected with an edge if their DCJ distance is $1$. The formula for the DCJ distance of two genomes can be given in terms of the numbers of cycles and even paths in their breakpoint graph. In mathematical terms, letting $T:=\{\tau_1,\cdots,\tau_{2k}\}$ and $T':=\{\tau'_1,\cdots,\tau_{2k}'\}$, the breakpoint graph of two genomes $G_1,G_2\in \mathcal{G}_{n,k}$, denoted by $BP(G_1,G_2)$, is a graph whose vertices are given by $\{\pm 1,\cdots, \pm n\}\cup T\cup T'$ and whose edges are the adjacencies (edges) of $G_1$ and $G_2^*$ (multiple edges are allowed), where $G_2^*$ is obtained from $G_2$ by changing the labels of its telomeric vertices $\tau_j$ to $\tau'_j$, for $j=1,\cdots,2k$.  The adjacencies of $G_1$ and $G_2^*$ are usually presented in two different colors, say gray and black. We notice that every vertex of $BP(G_1,G_2)$ which is obtained from a telomeric vertex of $G_1$ or $G_2^*$ has degree $1$. The rest of the vertices of $BP(G_1,G_2)$ are all of degree $2$, that is each such vertex is incident to exactly one gray edge and one black edge. Hence, $BP(G_1,G_2)$ is the disjoint union of some alternating cycles and $2k$ alternating paths (sometimes called lines in this paper). The DCJ distance between two genomes $G_1$ and $G_2$ is then given by
\begin{equation}\label{dcjformula}
d(G_1,G_2)=n-\mathbf C(G_1,G_2)-\frac{\mathbf P_e(G_1,G_2)}{2},
\end{equation}
where $\mathbf C(G_1,G_2)$ and $\mathbf P_e(G_1,G_2)$ determine the number of cycles and the number of even paths in $BP(G_1,G_2)$, respectively. A line is called a $TT$-line or with $TT$ ends if both telomeric vertices of that line belong to $T$. Lines with $TT'$ or $T'T'$ ends are defined similarly. It is straightforward to see that a line is even (of even size) if and only if it has $TT'$ ends. In the same spirit of the standard direction assigned to linear and circular chromosomes of a genome in $\mathcal{G}_{n,k}$, one can give a natural direction to the cycles and lines of a breakpoint graph. This can be done by determining a start and an end point for a line. More precisely, the direction of a $TT'$ line is given by traversing it, starting at its $T$ telomeric vertex and ending at its $T'$ telomeric vertex. For $TT$ or $T'T'$ lines, the starting point is the telomeric vertex with a smaller index. On the other hand, a cycle can be traversed, starting at its gene extremity with the minimum index, through the unique black edge incident to it. It is clear that the described direction of a cycle or a line of $BP(G_1,G_2)$ does not necessarily coincide with its orientation induced from the direction of edges in genomes $G_1$ and $G_2$.\\

The effect of a DCJ operation on the distance between two genomes can be studied as follows. Let $G_1,G_2\in \mathcal{G}_{n,k}$ and $e=[a,b],e'=[c,d]\in E_2(G_2)$, where by $[a,b]=[a,b]_{G_1,G_2}$, we mean that the direction of $e$ is from $a$ to $b$, in the natural direction of $BP(G_1,G_2)$, described above. Note once again that this direction may be different from the direction of $\{a,b\}$ in the genome $G_2$. The DCJ operation $\Delta_1(e,e')=\Delta_1^{G_1,G_2}(e,e')$ is defined on $G_2$ by cutting $e$ and $e'$ and joining $a$ to $c$, and $b$ to $d$. Similarly $\Delta_2(e,e')=\Delta_2^{G_1,G_2}(e,e')$ cuts $e$ and $e'$ and joins $a$ to $d$, and $b$ to $c$. Note that $\delta_i$ and $\Delta_i$ operations are different due to the different orders they use to join vertices.\\

Let $\alpha_i:=d(G_1,G_2\circ \Delta_i(e,e'))-d(G_1,G_2)$. In fact, $\alpha_i$ measures the displacement of the genome $G_2$, with respect to a reference genome, after a $\Delta_i$-DCJ operation is performed on it. If $e,e'$ are in two different cycles of $BP(G_1,G_2)$, then $\alpha_1=\alpha_2=1$, since two cycles merge following the DCJ operation on $e$ and $e'$, as in Figure \ref{dcj}-$III$. If both are in one cycle (Figure \ref{dcj}-$IV$), then $\Delta_1$ does not change the number of cycles or paths, but $\Delta_2$ fragments it into two smaller cycles. Hence $\alpha_1=0$ and $\alpha_2=-1$.  If they belong to one line, then, similarly to the cycle case, $\alpha_1=0$. In this case, $\Delta_2$ splits the line into a cycle and a line whose length has the same parity as that of the original line (e.g. if the original line is of an even size, then so is the size of the new line). In this case, since a new cycle is added, we get $\alpha_2=-1$.  When one edge belongs to an odd line ($TT$ or $T'T'$ line) and the other belongs to an even line ($TT'$-line), as in Figure \ref{dcj}-$I$, clearly, after splitting and rejoining two lines together, both $\Delta_1$ and $\Delta_2$ operations result in an odd and an even lines, which means $\alpha_1=\alpha_2=0$. If both lines are odd, but all of their $4$ telomeric vertices come from one genome, say with $TT$ ends, then both types of DCJ operator generate two odd $TT$-lines, i.e. $\alpha_1=\alpha_2=0$. If both are odd, one with $TT$ ends and the other with $T'T'$-ends, then $\Delta_1$ and $\Delta_2$ operations generate two $TT'$-lines which are even, hence $\alpha_1=\alpha_2=-1$. For two even $TT'$-lines, the $\Delta_1$ operation gives two odd lines, one with $TT$ ends and the other with $T'T'$ ends, while $\Delta_2$ operator gives two even $TT'$-lines. Therefore, $\alpha_1=1, \alpha_2=0$. Finally, as shown in Figure \ref{dcj}-$II$, if $e$ belongs to a line and $e'$ belongs to a cycle, the cycle and line merge after implementing $\Delta_1$ or $\Delta_2$ operations. This gives rise to a line whose length has the same parity as that of the original one. Hence in this case $\alpha_1=\alpha_2=1$.
\begin{figure}[!h]
            \centering
            \includegraphics[width=.85\textwidth]{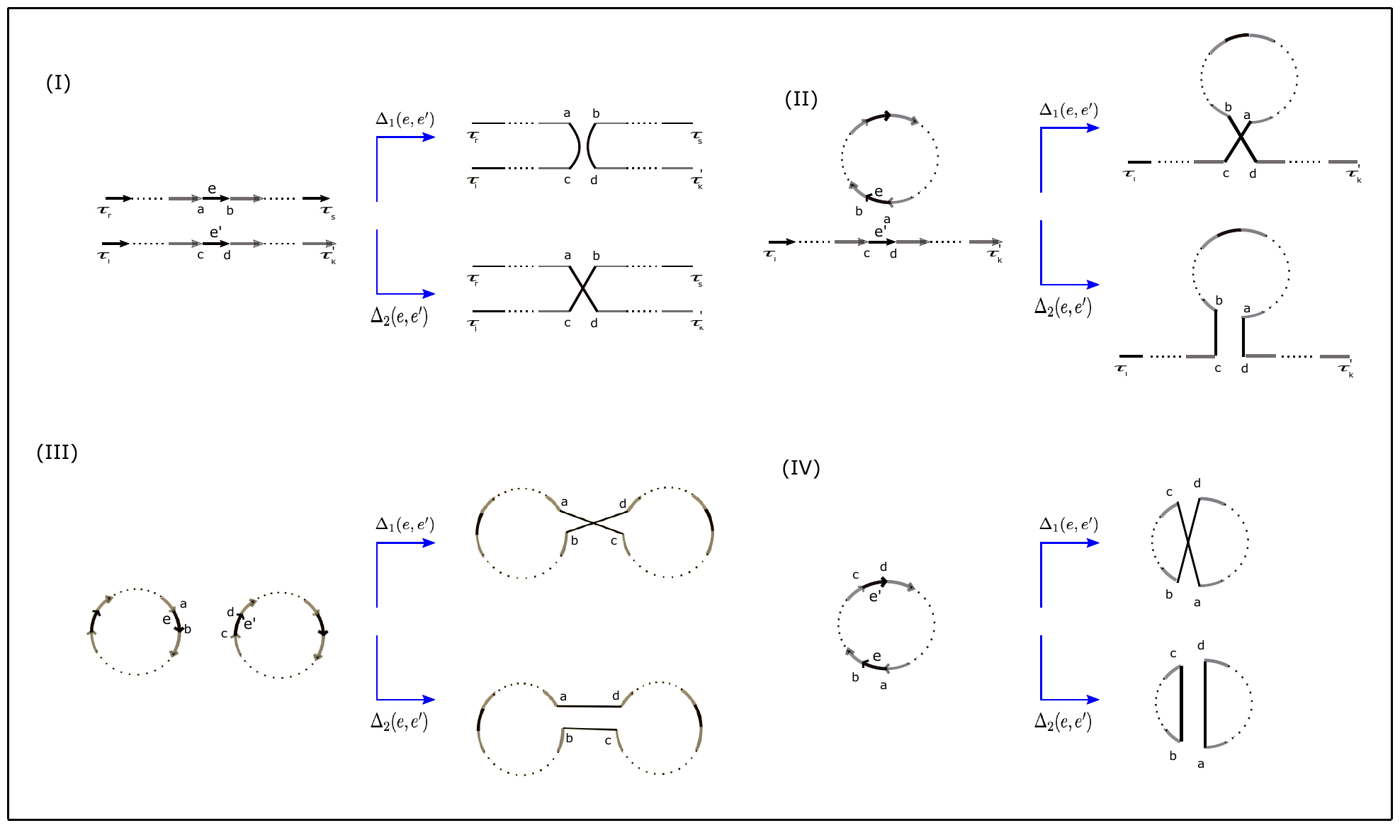}
            \caption{Effects of DCJ operations on the components of the breakpoint graph of two genomes.}  
            \label{dcj}
        \end{figure}
\section{DCJ evolution and DCJ parsimonious paths}
As mentioned in the previous section, the reversals are, in fact, DCJ-1 operations. DCJ can also generate translocations, fissions and fusions. In addition, two consecutive DCJ's may generate a genomic transposition or a block interchange operation \cite{yaf-2005, bms-2006, Tannier-book}. Consequently, DCJ's provide a rich family of mathematical operations which model the genomic evolution of some species and their edit distances. We can model the DCJ evolution of genomes as a Markov jump process on the space of genomes $\mathcal{G}_{n,k}$. Given an initial weight $p\in [0,1]$, let $\iota_1,\iota_2,\cdots$ be an \textit{i.i.d.} sequence of random variables with $\p(\iota_k=1)=p=1-\p(\iota_k=2)$, for any $k\in \N$. The DCJ process $X^{(p)}=(X^{(p)}_t)_{t\geq 0}$ starts at genome $X_0^{(p)}=G_0\in \mathcal{G}_{n,k}$, and jumps at random times $0\leq \tau_1<\tau_2<\cdots$, where $\{\tau_1,\tau_2,\cdots\}$ are points of a Poisson process on $\R_+$, with rate $1$. At a jump time $\tau_i$, having $X_{\tau_i}^{(p)}=G$, we choose two different solid edges $(a,b), (c,d)\in E_2(G)$, uniformly at random, without replacement, and let the random walk jump to $G\circ \delta_{\iota_i}((a,b),(c,d))$. In other words, at each time $\tau_i$, $X^{(p)}$ jumps to $G\circ\delta_1(e,e')$ with probability $p_{n,k}:=p/\{(n+k)(n+k-1)\}$, and to $G\circ\delta_2(e,e')$  with probability $q_{n,k}:=(1-p)/\{(n+k)(n+k-1)\}$, for any pair $e\neq e'\in E_2(G)$. Equivalently, $X^{(p)}$ can be considered as a continuous-time biased nearest-neighbour random walk on $\mathscr{G}_{n,k}$ starting at $G_0$, for which at a Poisson time $\tau_i$, it jumps to one of its neighbours with a probability proportional to its weight where the weights of each $\delta_1$-neighbour and each $\delta_2$-neighbour of $G$ are given by $p$ and $1-p$, respectively. We can see that, for $p=1/2$, $X^{(p)}$ is a continuous-time (symmetric) simple random walk on $\mathscr G_{n,k}$. Hereafter, we suppress the superscript $p$ when $p$ is fixed and there is no risk of ambiguity.\\

We can generalize the above $DCJ$ process as follows. Let $\mathbf{p}:=(p_t)_{t\geq 0}$, $p_t\in [0,1]$ for $t\in \R_+$, be  a given function from $\R_+$ to $[0,1]$. In order to obtain $X^{(\mathbf{p})}=(X^{(\mathbf{p})}_t)_{t\geq 0}$ from $X^{(p)}$, we modify the above definition by giving the weight $p_{\tau_k}$ and $1-p_{\tau_k}$ to $\delta_1$ and $\delta_2$ operations which can be performed on the current genome at each jump time $\tau_k$. Then, the probability of choosing each $\delta_1$ operation and each $\delta_2$ operation, at time $\tau_k$, are given by $p_{\tau_k}/\{(n+k)(n+k-1)\}$ and $(1-p_{\tau_k})/\{(n+k)(n+k-1)\}$. In fact, $X^{(p)}$ is a special case of $X^{(\mathbf{p})}$, for $p_t=p$; $t\in \R_+$.\\

Being parsimony-bound means that in transforming into another genome $B$, a genome $A$ has evolved along  a shortest path to $B$. This is, in fact, the fastest way that $A$ can be transformed into $B$ through the paths of evolution. Roughly speaking, any set of allowed genomic operations $R$ on a space of genomes $S$, under which the genome space is closed, induces an edit distance $\rho$ on $S$ that, for $G_1,G_2\in S$, is defined as the minimum number of elements of $R$ needed to transform $G_1$ into $G_2$ or vice-versa. Of course, we assume that this is possible in finite steps, for any pair of genomes in $S$. When a random walk $(\xi_t)_{t\geq 0}$ models the evolution of genomes under $R$-operations, evolution is parsimony-bound up to time $t$, if
\[
\rho(\xi_0,\xi_t)\approx t.
\]
Setting $S:=\mathcal{G}_{n,k}$, and letting $R$ include all DCJ operations, the simulation results in Figure \ref{sims} show that the parsimony binds the trajectory up to time $cn$, for $c\approx 0.5$. 
\begin{figure}[!h]
            \centering
            \includegraphics[width=.85\textwidth]{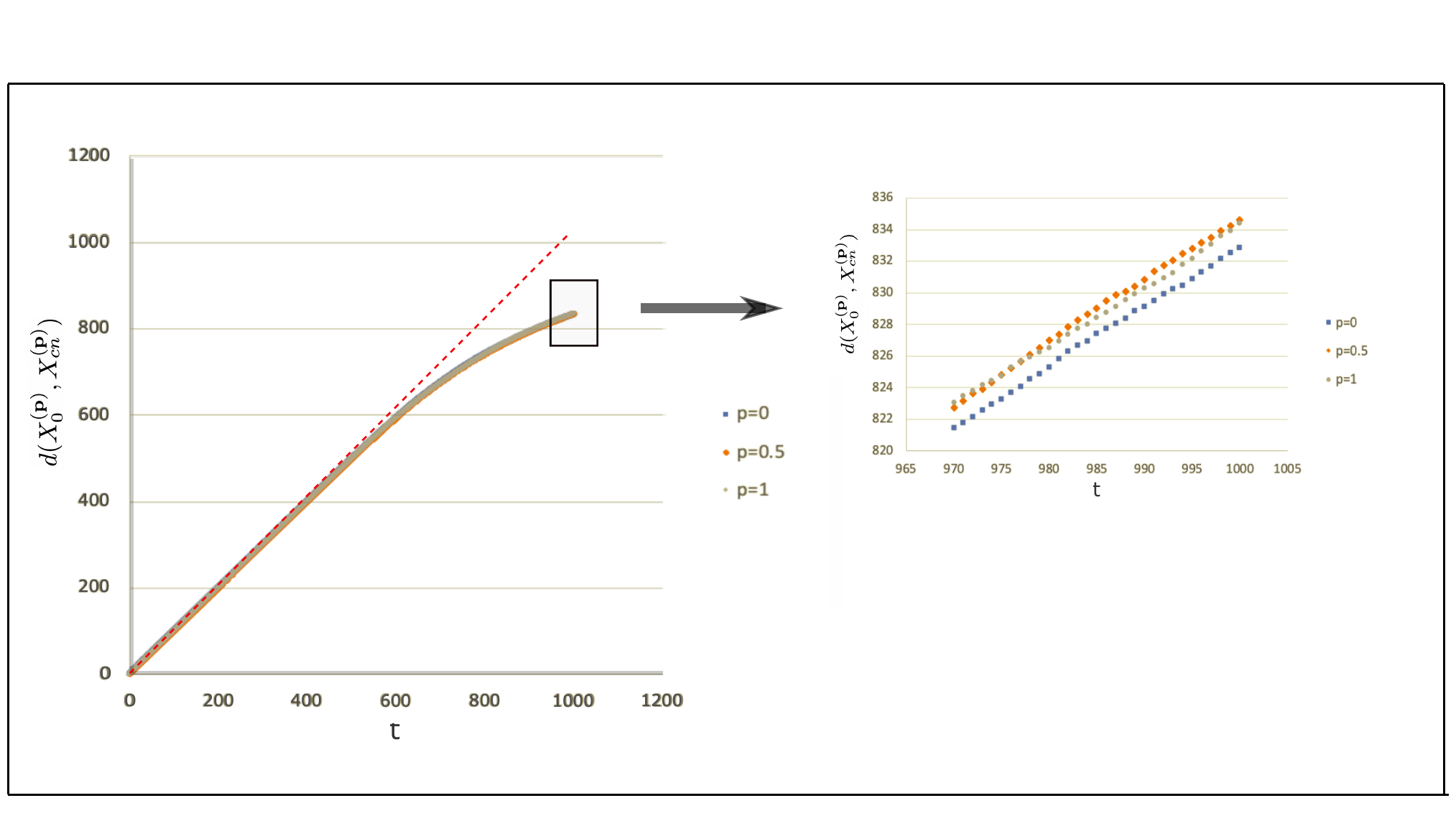}
            \caption{Simulation for unrestricted DCJ model illustrating parsimony binding up to time $cn$, for $c\approx 0.5$. The average DCJ distance estimates are based on $100$ runs, for $p=0, 0.5, 1$. The ancestral genome contains $1000$ genes in $4$ linear chromosomes of size $100, 200, 300$ and $400$.} 
            \label{sims}
        \end{figure}
        Berestycki and Durrett \cite{bd-2006}  proved that  the transposition random walk is bound to parsimony up to time $n/2$. Counting the number of hurdles of a reversal random walk, and applying the results of \cite{bd-2006} , one can see this also holds for the reversal walk up to the same time $n/2$ \cite{js-2013}. Here, we make use of the method developed in \cite{bd-2006} to show that the same result is true for the DCJ stochastic process. To this end, we need a relabeling mechanism that is updated at each jump of the DCJ process. This mechanism is important to construct a random graph process (see also \cite{bd-2006}) whose number of components estimates the number of cycles of the $BP(X_0,X_t)$, and therefore the DCJ distance between the genomes $X_0$ and $X_t$. A labeling of a genome $G\in \mathcal{G}_{n,k}$ is a bijection from $E_2(G)$ onto $[n+k]:=\{1, 2,\cdots, n+k\}$. The labeling process $L=(L_t)_{t\geq 0}$ can be obtained from the DCJ process $X$ as follows. We start from an arbitrary label $L_0=\hat{L}_0$ for $X_0$. From exchangeability, the results in the next theorems do not really depend on any particular choice of $\hat{L}_0$. Let $t$ be a jump time such that $X_t$ jumps to $X_{t+}:=X_t\circ \delta_i(e,e')$ for $e=(a,b),e'=(c,d) \in X_t$. Let $x:=\min\{a,b,c,d\}$, and let $\tilde{u}_x$ be the black edge of $X_t$, either $e$ or $e'$, which is incident to $x$, while $\tilde v_x$ be the other edge, i.e. $\tilde v_x=\{e,e'\}\setminus \tilde u_x$.  Similarly, let $u_x$ be the edge of $X_{t+}$ which is incident to $x$ after performing $\delta_i(e,e')$ on $X_t$, and let $v_x$ be the other new edge in $X_{t+}$ that is created as a result of joining the two other free gene extremities. We update the labeling process $L$ by 
\begin{equation}
L_{t+}(w) =
\left\{
	\begin{array}{lll}
		  L_t(\tilde{u}_x) & \mbox{if } \ w=u_x \\
		  L_t(\tilde{v}_x) & \mbox{if } \ w=v_x \\
		  L_t(w) & \mbox{if} \  \ w\notin\{u_x,v_x\}.
	\end{array}
\right.
\end{equation}
Using $L$ and $X$, we can now define a random graph process $Z=(Z_t)_{t\geq 0}$ whose number of tree components estimates the number of cycles of the breakpoint graph, after a convenient rescaling. Let $Z_0$ be the trivial graph with vertices $1,2,\cdots,n+k$ (representing the labels) and  with no edges.  Let $(\lambda_1(i),\lambda_2(i))$ be the labels of two uniformly random edges (without replacement) in $X_{\tau_i}$, which is selected for a DCJ operation at  time $\tau_i$. Connect $\lambda_1(i)$ and $\lambda_2(i)$ by an edge in $Z_{\tau_i+}$, if they are not already adjacent in $Z_{\tau_i}$.  As a matter of fact, the number of times at which the random sequence $((\lambda_1(i),\lambda_2(i)))_{i\in \N: \ \tau_i<cn}$ takes an ordered pair of labels $(\ell,\ell')\in [n+k]^2$, is a Poisson random variable with mean $cn/\{(n+k)(n+k-1)\}$. Hence, the probability that each edge $\{\ell,\ell'\}$ appears in the random graph is given by
\[
1-e^{-\frac{2cn}{(n+k)(n+k-1)}}\approx \frac{2cn}{(n+k)(n+k-1)}+O(n^{-2}).
\]
Thus, letting $Z^*_{cn}$ be the Erd\"os-R\'enyi random graph on the vertex set $[n+k]$ in which each edge is independently present with probability $2c/n$, the expected difference of number of edges in $Z_{cn}$ and $Z^*_{cn}$ is of order $O(1)$. Therefore, from \cite[Theorem 5.12]{bollobas-2001},
\[
\e[\mathcal{T}_{cn}]=(1-\gamma(c))n+O(1),
\]
where $\mathcal{T}_{cn}$ denotes the number of tree components of $Z_{cn}$, and, for $c>0$,
\[
\gamma(c):=1-\frac{1}{2c}\sum\limits_{j=1}^\infty \frac{j^{j-2}}{j!}(2c e^{-2c})^j.
\]
In fact, $\gamma(c)=c$, for $0<c\leq \frac{1}{2}$, and $\gamma(c)<c$, for $c>1$. Similarly, once again from \cite[Theorem 5.12]{bollobas-2001}, we get
\[
var(\mathcal{T}_{cn})=O(n).
\]
We can then easily see that the Erd\"os-R\'enyi random graph $Z^*_{cn}$ approximates $Z_{cn}$. In mathematical words, from the Chebyshev's inequality
\begin{equation}\label{random-graph-tree-convergence}
\frac{\mathcal{T}_{cn}-(1-\gamma(c))n}{c_n\sqrt{n}}\xrightarrow{p}0, \ \ \ n\rightarrow \infty,
\end{equation}
for an arbitrary sequence $(c_n)_{n\in \N}$, with $c_n\rightarrow \infty$, as $n\rightarrow \infty$.
\begin{theorem}\label{random-graph-thm}
Fix $\mathbf{p}=(p_t)_{t\geq 0}$ and $c>0$. Let $(c_n)_{n\geq 1}$ be an arbitrary sequence for which $c_n\rightarrow \infty$, as $n\rightarrow \infty$. Then, as $n\rightarrow \infty$,
\begin{equation}\label{DCJ-parsimony}
\frac{d(X_0^{(\mathbf{p})},X_{cn}^{(\mathbf{p})})-\gamma(c)n}{c_n\sqrt{n}} \xrightarrow{p}0.
\end{equation}
\end{theorem}
\begin{proof}
We follow the proof of \cite[Theorem 3]{bd-2006} and adapt it for the DCJ stochastic process. First, note that the number of paths of even length in $BP(X_0,X_{cn})$ is bounded by $2k$, and hence, $$\mathbf{P}_e(X_0,X_{cn})/\{2c_n\sqrt{n}\}\rightarrow 0,$$ pointwisely, as $n\rightarrow \infty$. Using (\ref{dcjformula}), it then suffices to analyze the asymptotic behavior of $\mathbf{C}(X_0,X_{cn})$. The number of cycles of size greater than $2\sqrt{n+k}$ is less than $\sqrt{n+k}$. Therefore, we only need to handle the cycles whose sizes are less than or equal to $2\sqrt{n+k}$. To this end, we say a cycle in $BP(X_0,X_{cn})$ is non-fragmented, if it is the result of a finite number of merging events, without any fragmentation in its history. Otherwise, it is called a fragmented cycle. Note that a fragmented cycle always has at least one fragmentation event (either a fragmentation on a cycle or on a line) in its history. For any $t\geq 0$, denote by $\mathcal{C}_t^{(f)}$ and $\mathcal{C}_t^{(nf)}$ the number of fragmented and non-fragmented cycles of $BP(X_0,X_t)$, respectively, with size less than or equal to $2\sqrt{n+k}$. Clearly, $\mathcal{C}_{cn}^{(f)}$ is bounded by $2$ times the number of fragmentation events, up to time $cn$, which results in a fragmented cycle of size at most $2\sqrt{n+k}$. In order that a DCJ operation on two black edges $e,e'$, located in one connected component of $BP(X_0,X_t)$, splits that component and produces a cycle of size less than $2\sqrt{n+k}$, $e$ and $e'$ should be in distance at most $2\sqrt{n+k}$, in $BP(X_0,X_{cn})$, from each other. At each jump time $\tau_i$, regardless of the value of $p_{\tau_i}$, the probability that $(\lambda_1(i),\lambda_2(i))$ takes a pair of edges with the mentioned property is bounded by $2\sqrt{n+k}/\{n+k-1\}$. Therefore, the number of such fragmentation events, up to time $cn$, is bounded by a Poisson random variable with parameter $cn\cdot 2\sqrt{n+k}/\{n+k-1\}\approx 2c\sqrt{n}$. This implies that, as $n\rightarrow \infty$,
\[
\frac{\mathcal{C}_{cn}^{(f)}}{c_n\sqrt{n}}\xrightarrow{p} 0.
\]
To estimate $\mathcal{C}_{cn}^{(nf)}$, for any pair of labels $\{\ell,\ell'\}$, let $I_{cn}(\ell,\ell')=I_{cn}(\ell', \ell)=1$, if the random sequence $(\{\lambda_1(i),\lambda_2(i)\})_{i: \ \tau_i<cn}$ takes $\{\ell,\ell'\}$ at least twice, and let it be equal to $0$, otherwise. We have
\[
\e[I_{cn}(\ell,\ell')]=\p(I_{cn}(\ell,\ell')=1)=e^{-\frac{2cn}{(n+k)(n+k-1)}}\sum\limits_{j=2}^\infty \left(\frac{2cn}{(n+k)(n+k-1)}\right)^j/j!=O(n^{-2}),
\]
and
\[
var(I_{cn}(\ell,\ell'))=\p(I_{cn}(\ell,\ell')=1)-\p(I_{cn}(\ell,\ell')=1)^2=O(n^{-2})
\]
Letting $I_{cn}=\sum_{\ell,\ell'} I_{cn}(\ell,\ell')$, we have $\e[I_{cn}]=O(1)$ and $var(I_{cn})=O(1)$, that means the number of edges in $Z_{cn}$ for which the corresponding labels  are taken more than once is negligible, and in fact, as $n\rightarrow \infty$,
\[
\frac{I_{cn}}{c_n\sqrt{n}}\xrightarrow{p} 0.
\]
We notice that the tree components of $Z_{cn}$ represent either the non-fragmented cycles and paths of $X_{cn}$ or those fragmented cycles and paths of $X_{cn}$ for which a same pair of labels for a DCJ operation is selected more than once (the number of these cycles or paths is bounded by $I_{cn}$). As the number of even paths is bounded by $2k$, and the number of cycles or tree components of size greater than $2\sqrt{n+k}$ are both bounded by $\sqrt{n+k}$, the tree components of $Z_{cn}$ can approximate $\mathcal{C}_{cn}^{(nf)}$, and
\[
\frac{\mathcal{T}_{cn}-\mathcal{C}_{cn}^{(nf)}}{c_n\sqrt{n}}\xrightarrow{p} 0,
\]
as $n\rightarrow \infty$. Therefore (\ref{random-graph-tree-convergence}) yields
\[
\frac{\mathcal{C}_{cn}^{(nf)}-(1-\gamma(c))n}{c_n\sqrt{n}}\xrightarrow{p} 0,
\]
as $n\rightarrow \infty$, which finishes the proof.
\end{proof}
\begin{remark}
For $c\leq 0.5$, the convergence in (\ref{DCJ-parsimony}) implies that  DCJ stochastic evolution $X$ is bound by parsimony, since $\gamma(c)=c$.
\end{remark}
\section{Restricted DCJ evolution and parsimonious paths}
In the previous section we studied the parsimony binding for the DCJ model when there is no constraint for the number of circular genomes. In other words, an arbitrary number of $\delta_2$-operations, say $r$, could be consecutively performed on one or more linear chromosomes in order to produce $r$ new circular chromosomes in the genome. In this section we study the parsimony problem for the DCJ operations restricted to the unichromosomal linear genomes with the same genes, that is we assume that consecutive $\delta_1$-operations (reversals) are allowed without any restriction, but on the contrary, no two consecutive $\delta_2$-operations are allowed, which means that once an intermediate circular chromosome appears as a result of a $\delta_2$-operation, two adjacencies (black edges), one from each chromosome, are chosen and a DCJ on them generates a unichromosomal linear genome. Another way of looking into this, is by considering a particular subgraph of $\mathcal{G}_{n,1}$, denoted by $\mathcal{G}_n^*$, induced by those vertices (genomes) in $\mathcal{G}_{n,1}$ which have at most one circular chromosome. The vertices of $\mathcal{G}_n^*$ can be partitioned into $U$, the unichromosomal linear genomes, and $\tilde{U}$, the genomes with exactly one linear and one circular chromosomes. From the definition, it is clear that no two vertices in $\tilde{U}$ are adjacent in $\mathcal{G}_n^*$. The edges between two vertices in $U$ represent reversals, while the ones between a vertex from $U$ and a vertex from $\tilde{U}$  represent the non-reversal DCJ operations needed to transform a unichromosomal linear genome into a genome with one linear and one circular chromosome, and vice versa. The restricted DCJ distance $d^*$ between two genomes $G_1,G_2\in \mathcal{G}_n^*$ is then the shortest path between $G_1$ and $G_2$ in $\mathcal{G}_n^*$. Yancopoulos et al. \cite{yaf-2005} showed that, in fact, for any pair of genomes $G_1,G_2\in \mathcal{G}_n^*$,
\[
d^*(G_1,G_2)=d(G_1,G_2).
\]
Hence, once again, one can use the number of cycles and even paths in $BP(G_1,G_2)$, to find $d^*(G_1,G_2)$.\\
 
For $\mathbf{p}=(p_t)_{t\in \R_+}$, where $0\leq p_t\leq 1$, $t\in \R_+$, we define a continuous-time biased nearest-neighbour random walk $Y^{(\mathbf{p})}=(Y^{(\mathbf{p})}_t)_{t\geq 0}$ on $\mathcal{G}_n^*$, where $Y^{(\mathbf{p})}$ jumps at Poisson times with rate $1$. At each jump time $t$, if $Y^{(\mathbf{p})}_t\in \tilde{U}$, then it chooses one of its neighbours, uniformly at random, and jumps to it, and if $Y^{(\mathbf{p})}_t\in U$, then it jumps to each of its $\delta_1$-neighbours (obtained as a result of a reversal DCJ) with probability $2p_t/\{n(n+1)\}$ and jumps to each of its $\delta_2$-neighbours (obtained as a result of a non-reversal DCJ) with probability $2(1-p_t)/\{n(n+1)\}$. The case $p_t=1$, for any $t\geq 0$, gives the reversal (inversion) random walk on the symmetric group of size $n$. It is known that the reversal random walk escapes from its origin at its maximum speed up to time $n/2$, for large $n$, where $n$ is the number of genes in a unichromosomal linear genome \cite{bd-2006, js-2013}.\\ 

We define the labeling $\tilde{L}=(\tilde{L}_t)_{t\geq 0}$ and the random graph process $\tilde{Z}=(\tilde{Z}_t)_{t\geq 0}$, from $Y^{(\mathbf{P})}$, in the same manner as we defined $L$ and $Z$ from $X^{(\mathbf{P})}$. As before, we drop the superscript $\mathbf{p}$ when there is no cause for confusion. The following theorem estimates $d(Y_0,Y_{cn})$ using the number of tree components of $\tilde{Z}_{cn}$.
\begin{theorem}
Let $\mathbf{p}=(p_t)_{t\geq 0}$, such that $0\leq p_t\leq 1$, $c>0$, and let $(c_n)_{n\geq 1}$ be an arbitrary sequence of real numbers such that $c_n\rightarrow \infty$, as $n\rightarrow \infty$. Then, as $n\rightarrow \infty$,
\[
\frac{d(Y_0^{(\mathbf{p})},Y_{cn}^{(\mathbf{p})})-n+\tilde{\mathcal{T}}_{cn}}{c_n\sqrt{n}}\xrightarrow{p} 0,
\]
where $\tilde{\mathcal{T}}_t$ denotes the number of tree components of $\tilde{Z}_t$.
\end{theorem}
\begin{proof}
The proof is similar to the proof of Theorem \ref{random-graph-thm}. We define the number of fragmented and non-fragmented cycles of $\tilde{Z}_t$, as we defined the similar notions for $Z_t$, and denote by $\tilde{\mathcal{C}}_t^{(f)}$ and $\tilde{\mathcal{C}}_t^{(nf)}$, the number of fragmented and non-fragmented cycles of $\tilde{Z}_t$ whose sizes are not greater than $2\sqrt{n+1}$. As before, there are at most  $2$ paths and at most $\sqrt{n}$ cycles of size greater than $2\sqrt{n+1}$  in $BP(Y_0,Y_{cn})$. Therefore, we only need to handle the asymptotic behaviors of $\tilde{\mathcal{C}}_t^{(f)}$ and $\tilde{\mathcal{C}}_t^{(nf)}$. Once again, the former is bounded by $2\zeta$, where $\zeta\sim Po(2c\sqrt{n+1})$, and hence, as $n\rightarrow 0$,
\[
\frac{\tilde{\mathcal{C}}_{cn}^{(f)}}{c_n\sqrt{n}}\xrightarrow{p} 0.
\]
Now , it suffices to show, as $n\rightarrow \infty$,
\[
\frac{\tilde{\mathcal{T}}_{cn}-\tilde{\mathcal{C}}_{cn}^{(nf)}}{c_n\sqrt{n}}\xrightarrow{p} 0.
\]
To establish this, for any pair of labels $\{\ell,\ell'\}$ in $(BP(\tilde{Z}_0,\tilde{Z}_{t}))_{0\leq t\leq cn}$, we define the indicator functions $\tilde{I}_{cn}(\ell,\ell')=\tilde{I}_{cn}(\ell',\ell)$ and their total sum $\tilde{I}_{cn}$,  similarly to their counterparts introduced in the proof of Theorem \ref{random-graph-thm}. To find an upper bound for $\e[\tilde{I}_{cn}]$, first observe that the probability that any pair of labels $\{\ell,\ell'\}$ is picked for a DCJ at a jump time $t$ is given by $2/\{n(n+1)\}$, if $Y_t\in U$. If $Y_t\in \tilde{U}$, given that the size of the circular chromosome is $2j$(which means there are exactly $j$ black edges in the circular chromosome), this probability is given by $1/\{j(n+1-j)\}$. Letting $q_j$ be the probability that a non-reversal DCJ results in a circular chromosome of size $2j$, this means that the probability that any pair of labels $\{\ell,\ell'\}$ is picked for a DCJ, at time $t$, for $Y_t\in \tilde{U}$ is bounded by
\begin{equation}
\begin{split}
\sum\limits_{j=1}^n \frac{q_j}{j(n+1-j)}\leq & \ \frac{2}{n}\sum\limits_{j=1}^n \frac{1}{j(n+1-j)}\\
=&\frac{2}{n(n+1)}\sum\limits_{j=1}^n \left(\frac{1}{j}+\frac{1}{n+1-j}\right)\\
=& \frac{4 h_n}{n(n+1)} \approx \frac{4\ln n}{n(n+1)},
\end{split}
\end{equation}
where $h_n:=\sum_{j=1}^n 1/j$ is the harmonic number. Therefore, integrating this on $[0,cn]$ and letting $\tilde{\zeta}\sim Po(q_{cn}^*)$ for
\[
q_{cn}^*:=cn\left(\frac{2}{n(n+1)} \vee \frac{4h_n}{n(n+1)}\right)\approx \frac{4c \ln n}{n},
\]
we can couple the random variables $I_{cn}(\ell,\ell')$ with indicator random variables $J_{cn}(\ell,\ell')$ defined independently for any pair of labels $\{\ell,\ell'\}$ by 
\[
\p(J_{cn}(\ell,\ell')=1)=1-\p(J_{cn}(\ell,\ell')=0)=\p(\tilde{\zeta}\geq 2),
\] 
such that $J_{cn}(\ell,\ell')=1$ if $I_{cn}(\ell,\ell')=1$ (the converse is not assumed), i.e. $I_{cn}(\ell,\ell')\leq J_{cn}(\ell,\ell')$. We have
\[
\begin{split}
\p(J_{cn}(\ell,\ell')=1)=& e^{-q_{cn}^*} \sum\limits_{j=2}^\infty (q_{cn}^*)^j/j!\\
=& O\left(\left(\frac{\ln n}{n}\right)^2\right).
\end{split}
\]
Defining $J_{cn}=\sum_{\ell,\ell'} J_{cn}(\ell,\ell')$, we get $\e[J_{cn}]=O(\ln^2 n)$ and $var(J_{cn})=O(\ln^2 n)$. Thus, by Chebyshev's inequality, $J_{cn}/\{c_n\sqrt{n}\}\rightarrow 0$, in probability, which implies
\[
\frac{\tilde{I}_{cn}}{c_n\sqrt{n}}\xrightarrow{p} 0, \ \ n\rightarrow \infty,
\]
since $0\leq \tilde{I}_{cn}\leq J_{cn}$. The same lines of arguments as those in the proof of Theorem \ref{random-graph-thm} finish the proof.
\end{proof}
\begin{remark}
Note that $\tilde{Z}_{cn}$ is not an Erd\"os-R\'enyi random graph. Therefore, an exact analysis of the number of trees components of $\tilde{Z}_{cn}$ is not easy and needs further investigation.
\end{remark}
So far, we have seen that the number of cycles in $BP(Y_0,Y_{cn})$ can be estimated by the number of tree components of the random graph $\tilde{Z}_{cn}$. This provides a method to see to what extent evolution by the restricted DCJ process  is bound by parsimony. Of course, this could also be checked by simulating $Y$ and deriving the DCJ distance between $Y_{cn}$ and $Y_0$, using the existing linear algorithms for computing the DCJ distance (cf. \cite{yaf-2005, bms-2006}). But to make use of the method described for the Erd\"os-R\'enyi random graph, we investigated the maximum time for which evolution is bound to parsimony  by counting the number of tree components in $\tilde{Z}_{cn}$. The simulation study summarized in Figure \ref{sims2}) shows that the binding holds up to time $n/2$, for large $n$.
\begin{figure}[!h]
            \centering
            \includegraphics[width=.85\textwidth]{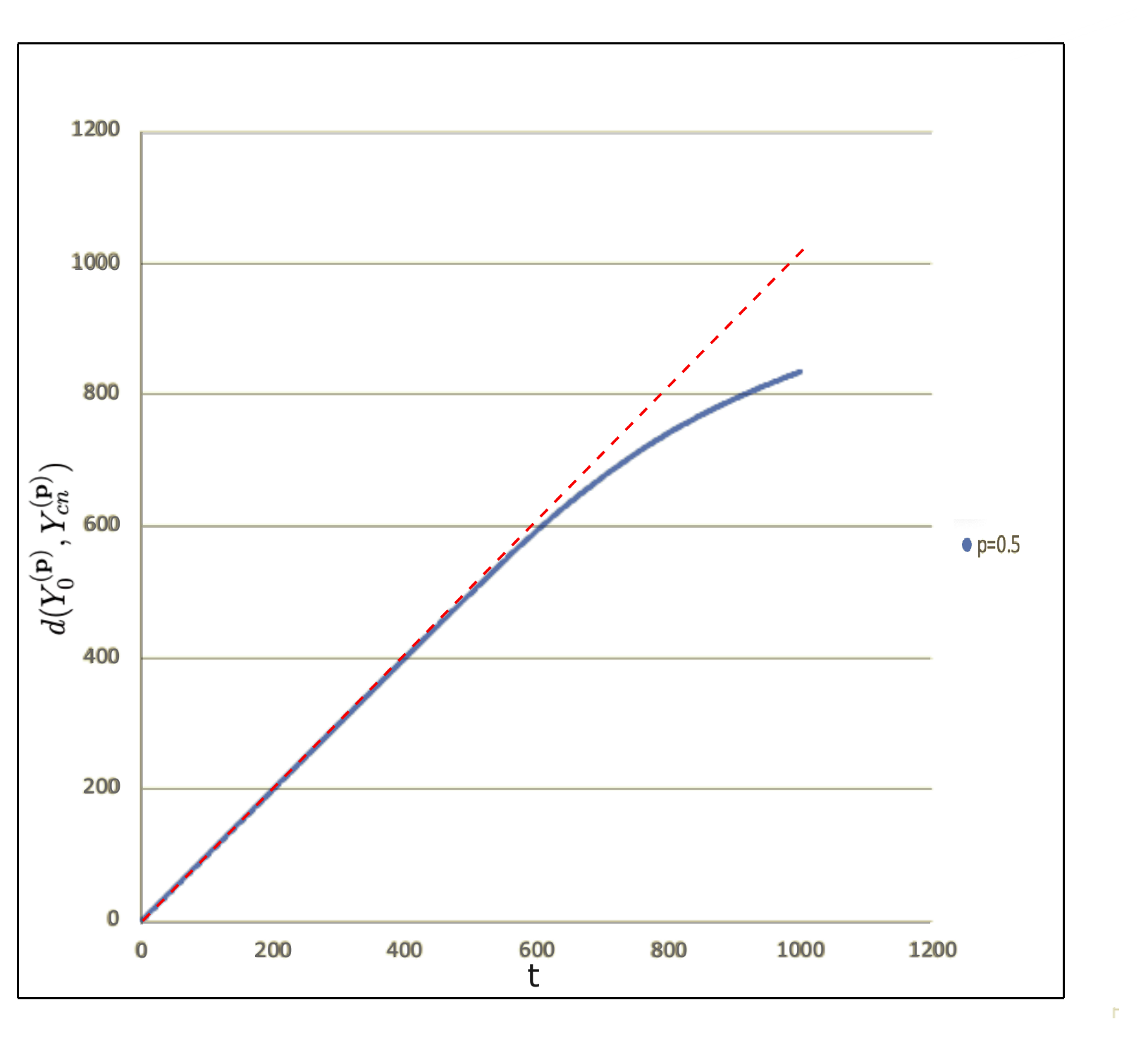}
            \caption{Simulation for the restricted DCJ model illustrating parsimony binding for restricted DCJ up to time $cn$, for $c\approx 0.5$. The average DCJ distance estimates are based on $100$ runs, for $1000$ genes; $p=0.5$. } 
            \label{sims2}
        \end{figure}

\bibliography{ref}
\bibliographystyle{unsrt}
\end{document}